\documentclass[12pt, reqno]{amsart}
\usepackage{amsmath, amsthm, amscd, amsfonts, amssymb, graphicx, color}
\usepackage[bookmarksnumbered, colorlinks, plainpages]{hyperref}

\input{mathrsfs.sty}

\newtheorem{theorem}{Theorem}[section]
\newtheorem{lemma}[theorem]{Lemma}

\newtheorem{corollary}[theorem]{Corollary}
\theoremstyle{definition}

\theoremstyle{remark}
\newtheorem{remark}[theorem]{Remark}

\begin{document}
\title[On the binary relation $\leq_u$]{On the binary relation $\leq_u$ on self-adjoint Hilbert space operators}
\author[M.S. Moslehian, S.M.S. Nabavi Sales, H. Najafi]{M. S. Moslehian, S. M. S. Nabavi Sales and H. Najafi}
\address{$^1$ Department of Pure Mathematics, Center of Excellence in
Analysis on Algebraic Structures (CEAAS), Ferdowsi University of
Mashhad, P. O. Box 1159, Mashhad 91775, Iran.}
\email{moslehian@ferdowsi.um.ac.ir and moslehian@member.ams.org}
\email{sadegh.nabavi@gmail.com} \email{hamednajafi20@gmail.com}

\keywords{Operator inequality; binary order; operator convex
function; hyponormal; strong operator topology.}

\begin{abstract}
Given self-adjoint operators $A, B\in\mathbb{B}(\mathscr{H})$ it is
said $A\leq_uB$ whenever $A\leq U^*BU$ for some unitary operator
$U$. We show that $A\leq_u B$ if and only if
$f(g(A)^r)\leq_uf(g(B)^r)$ for any increasing operator convex
function $f$, any operator monotone function $g$ and any positive
number $r$. We present some sufficient conditions under which if
$B\leq A\leq U^*BU$, then $B=A=U^*BU$. Finally we prove that if
$A^n\leq U^\ast A^nU$ for all $n\in\mathbb{N}$, then $A=U^\ast AU$.
\end{abstract}
\maketitle

\section{Introduction}

Let $\mathbb{B}(\mathscr{H})$ be the algebra of all bounded linear
operators on a complex Hilbert space $\mathscr{H}$ with the identity
$I$, let $\mathbb{B}_h(\mathscr{H})$ be the real linear space of all
self-adjoint operators and let $\mathcal{U}(\mathscr{H})$ be the set
of all unitary operators in $\mathbb{B}(\mathscr{H})$. By an
orthogonal projection we mean an operator $P \in
\mathbb{B}_h(\mathscr{H})$ such that $P^2=P$. An operator $A \in
\mathbb{B}(\mathscr{H})$ is called positive if $\langle Ax,x\rangle
\ge 0$ for every $x\in \mathscr{H}$ and then we write $A\ge 0$. If
$A$ is a positive invertible operator we write $A>0$. For $A, B \in
\mathbb{B}_h(\mathscr{H})$ we say that $A\le B$ if $B-A \ge 0$. The
celebrated L\"owner--Heinz inequality asserts that the operator
inequality $T\geq S\geq 0$ implies $T^\alpha \geq S^\alpha$ for any
$\alpha\in[0,1]$, see \cite[Theorem 3.2.1]{FUR}. An operator $T$ is
called hyponormal if $T^*T\geq TT^*$.

\noindent Douglas \cite{DOU} investigated the operator inequality
$T^\ast HT\leq H$, with $H$ Hermitian and showed that if $P$ is a
positive compact operator and $A$ is a contraction such that $P \leq
A^*PA$, then $P=A^*PA$; see also \cite{DUG}. Ergodic properties of
the inequality $T^\ast AT\leq A$, with $A$ positive studied by Suciu
\cite{SUC}.

\noindent Given operators $A, B\in \mathbb{B}_h(\mathscr{H})$ it is
said that $A\leq_uB$ whenever $A\leq U^*BU$ for some $U \in
\mathcal{U}(\mathscr{H})$; see \cite{KOS, OKU}. This binary relation
was investigated by Kosaki \cite{KOS} by showing that
\begin{eqnarray}\label{E1}
A\leq_uB \Rightarrow e^A\leq_ue^B.
\end{eqnarray}
Okayasu and Ueta \cite{OKU} gave a sufficient condition for a triple
of operators $(A,B,U)$ with $A, B\in \mathbb{B}_h(\mathscr{H})$ and
$U \in \mathcal{U}(\mathscr{H})$ under which $B\leq A\leq U^*BU$
implies $B=A=U^*BU$. In this note we use their idea and prove a
similar result. In fact we present some sufficient conditions on an
operator $U\in \mathcal{U}(\mathscr{H})$ for which $B\leq A\leq
U^*BU$ ensures $B=A=U^*BU$ when $A, B\in \mathbb{B}_h(\mathscr{H})$.
It is known that $\leq_u$ satisfies the reflexive and transitive
laws but not the antisymmetric law in general; cf. \cite{OKU}. The
antisymmetric law states that
$$A\leq_uB\,\,\,\mbox{and}\,\,\,B\leq_uA\Rightarrow A, B
\,\,\,\mbox{are unitarily equivalent}.$$ We, among other things,
study some cases in which the antisymmetric law holds for the
relation $\leq_u$. We refer the reader to \cite{FUR} for general
information on operators acting on Hilbert spaces. Utilizing a
result of \cite {MON} we show that $A\leq_u B$ if and only if
$f(g(A)^r)\leq_u f(g(B)^r)$ for any increasing operator convex
function $f$, any operator monotone function $g$ and any positive
number $r$. Recall that a real function $f$ defined on an interval
$J$ is said to be operator convex if $f(\lambda
A+(1-\lambda)B)\leq\lambda f(A)+(1-\lambda)f(B)$ for any $A, B\in
\mathbb{B}_h(\mathscr{H})$ with spectra in $J$ and $\lambda\in
[0,1]$ and is called operator monotone if $f(A)\leq f(B)$ whenever
$A\leq B$ for any $A, B\in \mathbb{B}_h(\mathscr{H})$ with spectra
in $J$, see \cite{UCH}. Finally we prove that if $A$ is a positive
operator and $U \in \mathcal{U}(\mathscr{H})$ such that $A^n\leq
U^\ast A^nU$ for all $n\in\mathbb{N}$, then $A=U^\ast AU$.

\section{The results}

First we give the following lemmas that we need in the sequel. The first one is applied frequently without referring to it.

\begin{lemma}\label{U1}
Let $A\in \mathbb{B}_h(\mathscr{H})$ and $U\in
\mathcal{U}(\mathscr{H})$. Then $f(U^*AU)=U^*f(A)U$ for any function
$f$ which is continues on the spectra of $A$.
\end{lemma}
\begin{proof}
First we note that $(U^*AU)^n=U^*A^nU$ for all $n$. Using the functional calculus and a sequence of polynomials uniformly converging to $f$ on ${\rm sp}(A)$, we conclude that $f(U^*AU)=U^*f(A)U$.
\end{proof}

\begin{lemma}\cite[Theorems 2.1, 2.3]{MNS}\label{UV1}
Let $T\in\mathbb{B}(\mathscr{H})$ be hyponormal and $T=U|T|$ be the
polar decomposition of $T$ such that $U^{n_0}=I$ for some positive
integer $n_0$, $U^{*n}\to I$ as $n \to \infty$ or $U^n\to I$ as
$n \to \infty$, where the limits are taken in the strong operator
topology. Then $T$ is normal.
\end{lemma}

\begin{remark} We note that if $U\in \mathcal{U}(\mathscr{H})$, then
$$\|U^n\xi-\xi\|=\|U^{*n}\xi-\xi\|\:\:\:\:(\xi\in\mathcal{H}\:\:\mbox{and}\:\:n\in\mathbb{N}).$$
Thus $U^{*n}\to I$ as $n\to\infty$ if and only if $U^n\to I$ as
$n\to\infty$, where all limits are taken in the strong operator
topology.\end{remark}

\begin{lemma}\label{UV2}
Let $U, V \in \mathcal{U}(\mathscr{H})$ be two commuting operators
such that $U^n\to I$ and $V^n\to I$ as $n \to \infty$. Then
$(UV)^n\to I$ as $n\to\infty$, where all limits are taken in the
strong operator topology.
\end{lemma}
\begin{proof}
It immediately follows from the following
$$\|(UV)^n\xi-\xi\|\leq\|(UV)^n\xi-U^n\xi\|+\|U^n\xi-\xi\|\leq\|V^n\xi-\xi\|+\|U^n\xi-\xi\|\:\:\:\:(\xi\in\mathscr{H}).$$
\end{proof}
\begin{theorem}\label{UV3}
Let $U\in \mathcal{U}(\mathscr{H})$ such that any one of the
following conditions holds:

({\rm i}) $U^{n_0}=I$ for some positive integer $n_0$,

({\rm ii}) $U^n\to I$ as $n\to\infty$ in which the limit is taken in
the strong operator topology.\\
Then $B\leq A\leq U^*BU$ implies that $B=A=U^*BU$ for any $A, B\in
\mathbb{B}_h(\mathscr{H})$.
\end{theorem}
\begin{proof}
Let $A, B\in \mathbb{B}_h(\mathscr{H})$ such that $B\leq A\leq
U^*BU$. There exist $\lambda>0$ such that $B+\lambda>0$. Put
$T=(B+\lambda)^{1\over2}U$. By our assumption we have

\begin{eqnarray}\label{1}
TT^*=B+\lambda\leq A+\lambda\leq U^*(B+\lambda)U= T^*T.
\end{eqnarray}
Thus $T$ is a hyponormal operator. Obviously
$|T|=U^*(B+\lambda)^{1\over2}U=U^*T$. Let $T=V|T|$ be the
polar decomposition of $T$. Hence $T=VU^*T$. It follows from the
invertibility of $T$ that $I=VU^*$, that is, $U=V$. Thus $T$
satisfies the conditions of Lemma \ref{UV1}. Therefore $T$ turns out to
be normal. Then (\ref{1}) yields that $B=A=U^*BU$.
\end{proof}
\begin{corollary}
Let $U, V \in \mathcal{U}(\mathscr{H})$ be two commuting operators
satisfying any one of the following conditions

({\rm i}) $U^{n_0}=I$ and $V^{n_0}=I$ for some positive integer
$n_0$,

({\rm ii}) $U^n\to I$ and $V^n\to I$ as $n\to\infty$,

where all limits are taken in the strong operator topology. If $A,
B\in \mathbb{B}_h(\mathscr{H})$ such that $A\leq U^*BU$ and $B\leq
V^*AV$, then $A=U^*BU$ and $B=V^*AV$.
\end{corollary}
\begin{proof}
By Lemma \ref{UV2}, the unitary operator $UV$ satisfies the
conditions of Theorem \ref{UV3}.
\end{proof}
The following lemmas are used in the proof of Theorem \ref{UV6}.
\begin{lemma}\cite[Theorem 3.2.3.1]{FUR}\label{UV4}
If $0< A\leq B$, then $\log(A)\leq\log(B)$.
\end{lemma}
The following result is a variant of Theorem 2.6 of \cite{MON}.
\begin{lemma}\cite[Theorem 2.6]{MON}\label{UV5}
Let $A, B \in \mathbb{B}(\mathscr{H})$ be two positive
operators.Then $B^2 \leq A^2$ if and only if for each operator
convex function $f$ on $[0,\infty)$ with $f^{'}_{+}(0) \geq 0$ it
holds that $f(B)\leq f(A)$.
\end{lemma}
If $f$ is an increasing operator convex function, then $f^{'}_{+}(0)
\geq 0$. The converse is also true. In fact $f$ can be represented
as
$$f(t)= f(0) + \beta t +\gamma t^2 +\int_{0}^{\infty}\frac{\lambda t^2}{\lambda+t} d\mu(\lambda)\,,$$
where $\gamma\geq 0$, $\beta = f^{'}_{+}(0)$ and $\mu$ is a positive
measure on $[0, \infty)$; see \cite [Chapter V]{BH}. Hence if
$f^{'}_{+}(0)$, then
$$f^{'}(t)= f^{'}_{+}(0)+2 \gamma t +\int_{0}^{\infty}\frac{2\lambda t + \lambda t^2}{(\lambda+t)^2} d\mu(\lambda) \geq 0$$
 for each $t\in [0, \infty)$. Now we are ready to state our next result.
\begin{theorem}\label{UV6}
Let $A$ and $B$ be two positive operators. Then $ A\leq_u B$ if and
only if $f(g(A)^r)\leq_u f(g(B)^r)$ for any increasing operator
convex function $f$, any operator monotone function $g$ and any
positive number $r$.
\end{theorem}
\begin{proof}
First we assume that $ 0<A\leq U^*BU$ for some operator $U\in
\mathcal{U}(\mathscr{H})$. Then $0<g(A)\leq g(U^*BU)=U^*g(B)U$ for
any operator monotone function $g$. Let $r$ be a positive number. By
Lemma \ref{UV4} we have $\log(g(A))\leq U^*\log(g(B))U$. Hence
$\log(g(A)^{2r})\leq \log (U^*g(B)^{2r}U)$. Thus by Kosaki result
(\ref{E1}) there is an operator $V\in \mathcal{U}(\mathscr{H})$ such
that $e^{\log g(A)^{2r}}\leq V^*e^{\log (U^*g(B)^{2r}U)}V$, that is
$g(A)^{2r}\leq V^*U^*g(B)^{2r}UV=(V^*U^* g(B)UV)^{2r}$. From which
and Lemma \ref{UV5} we conclude that
\begin{eqnarray}\label{3}
f(g(A)^r)\leq f((V^*U^*g(B)UV)^r)=V^*U^*
f(g(B)^r)UV
\end{eqnarray}
for any increasing operator convex
function $f$. This means that $f(g(A)^r)\leq_uf(g(B)^r)$ as desired.

For the general case note that the condition $ 0\leq A\leq U^* BU$
ensures $ 0< A+\varepsilon\leq U^* (B+\varepsilon) U$ for all
$\varepsilon>0$. Now the general result is deduced from the
paragraph above and a limit argument by letting $\varepsilon$ tend
to $0$.

The reverse is clear by taking $f(x)=x$ and $r=1$.
\end{proof}

From Theorem \ref{UV6} one can see that if $ A\leq_u B$, then there
exists a sequence $\{U_n\}_{n\in\mathbb{N}} \subset
\mathcal{U}(\mathscr{H})$ such that $A^n\leq U_n^*B^nU_n$. An
interesting problem is finding an operator $U\in
\mathcal{U}(\mathscr{H})$ such that $A^n\leq U^*B^nU$ for any
positive integer $n$. If there exist a sequence
$\{U_n\}_{n\in\mathbb{N}}\subset \mathcal{U}(\mathscr{H})$ such that
$A^n\leq U_n^*B^nU_n$ and in the strong operator topology $\{U_n\}$
converges to an operator $U\in \mathcal{U}(\mathscr{H})$, then $U$
is the desired unitary operator. To see this let
$\xi\in\mathscr{H},n\in\mathbb{N}$ and $m>n$. Then
\begin{eqnarray}\label{2}
\langle A^n\xi,\xi\rangle\leq\langle
U_m^*B^nU_m\xi,\xi\rangle=\langle
B^nU_m\xi,U_m\xi\rangle.
\end{eqnarray}
Note that in the inequality
of (\ref{2}) we used $\alpha={m\over n}$ in the L\"owner--Heinz inequality. By our assumption we have
$$\langle B^nU_m\xi,U_m\xi\rangle\rightarrow\langle
B^nU\xi,U\xi\rangle=\langle U^*B^nU\xi,\xi\rangle$$ as
$m\rightarrow\infty$, which by (\ref{2}) implies that $A^n\leq U^*
B^nU$ as requested.

The next theorem is related to the problem above. First we need to
introduce our notation. For any two positive operators $A$ and $B$
and any positive integer $n$ let
$\mathcal{K}_{n,A,B}=\{U\in\mathcal{U}(\mathscr{H}): A^n\leq
U^*B^nU\}$. This set is compact in the case when $\mathscr{H}$ is
finite dimensional. Further, $A\leq U^*BU$ for some unitary matrix
$U$ if $\lambda_j(A)\leq\lambda_j(B)\,\,(1\leq j\leq n)$, where
$\lambda_1(\cdot)\geq\ldots\geq\lambda_n(\cdot)$ denotes eigenvalues
arranged in the decreasing order with their multiplicities counted.
Thus $\mathcal{K}_{n,A,B}$ can be nonempty. Our next result reads as
follows.
\begin{theorem}\label{SS}
Suppose that $A$ and $B$ be two positive operators such that
$\mathcal{K}_{n_0,A,B}$ is a nonempty set, which is either compact
in the strong operator topology or closed in the weak operator
topology for some positive integer $n_0$. Then there exists an
operator $U\in \mathcal{U}(\mathscr{H})$ such that $A^n\leq U^*B^nU$
for every positive integer $n$.
\end{theorem}
\begin{proof}
First assume that $\mathcal{K}_{n_0,A,B}$ is a nonempty strongly
compact set for some positive integer $n_0$. Without loss of
generality we may assume that $n_0=1$. Let us set $\mathcal{K}_n$
instead of $\mathcal{K}_{n,A,B}$ for the sake of simplicity. Using
the L\"owner--Heinz inequality one easily see that for any positive
integer $n$\begin{eqnarray}\label{4} \mathcal{K}_{n+1}\subseteq
\mathcal{K}_{n}\,.
\end{eqnarray}
We show that the sets $\mathcal{K}_{n}$ are strongly closed. To
achieve this aim, fix $n$ and let $\{U_\alpha\}$ be a net in
$\mathcal{K}_n$ such that $U_\alpha \rightarrow U$ in which the
limit is taken in the strong operator topology. Since
$\mathcal{K}_n\subseteq\mathcal{K}_1$ and $\mathcal{K}_1$ is assumed
to be a strongly compact set, we conclude that $U\in\mathcal{K}_1$
which implies that $U\in \mathcal{U}(\mathscr{H})$. Let
$\xi\in\mathscr{H}$. We have
\begin{eqnarray}\label{5}
\langle A^n\xi,\xi\rangle\leq\langle U_\alpha^*
B^nU_\alpha\xi,\xi\rangle=\langle B^nU_\alpha\xi,U_\alpha\xi\rangle
\end{eqnarray}
Since $\{U_\alpha\}$ converges strongly to $U$ we obtain
\begin{eqnarray}\label{52}
\langle B^nU_\alpha\xi,U_\alpha\xi\rangle\rightarrow\langle
B^nU\xi,U\xi\rangle=\langle U^*B^nU\xi,\xi\rangle\,.
\end{eqnarray}
Applying (\ref{5}) and (\ref{52}) we get $A^n\leq U^* B^nU$. Thus
$U\in\mathcal{K}_n$. Hence $\mathcal{K}_n$ is closed. Now Theorem
\ref{UV6} shows that the sets $\mathcal{K}_n$ are nonempty and
(\ref{4}) shows that
$\bigcap_{n\in{F}}\mathcal{K}_n=\mathcal{K}_{\max{F}}\neq\phi$ for
any arbitrary finite subset ${F}$ of $\mathbb{N}$. Hence
$\bigcap_{n\in\mathbb{N}}\mathcal{K}_n\neq\phi$ because the
$\mathcal{K}_{n}$ are closed subsets of $\mathcal{K}_1$ and
$\mathcal{K}_1$ is compact.

Second, assume that $\mathcal{K}_{n_0}$ is a weakly closed nonempty
set for some positive integer $n_0$. Due to the unit ball of
$\mathbb{B}(\mathscr{H})$ is weakly compact, we can repeat the first
argument and reach to the desired consequence.
\end{proof}

Now we aim to prove our last result. We state some lemmas which are
interesting on their own right.

\begin{lemma}\label{U11}
Let $P\in\mathbb{B}(\mathscr{H})$ be an orthogonal projection and
$U\in \mathcal{U}(\mathscr{H})$ such that $P\leq U^\ast PU$. Then
$P=U^\ast PU$.
\end{lemma}

\begin{proof}
Let $ran(P)=\mathscr{H}_1$ and let $I_1$ and $I_2$ be the identity
operators on $\mathscr{H}_1$ and $\mathscr{H}_1^\bot$, respectively.
Therefore $P=I_1\oplus 0$  and $U=\left(
          \begin{array}{cc}
            U_1 & U_2 \\
           U_3 & U_4 \\
          \end{array}
        \right)\,\,\, \mbox{on\,}\,\mathscr{H}=\mathscr{H}_1\oplus
        \mathscr{H}_1^\bot$. From $P\leq U^\ast PU$ we reach to the
following inequality
\begin{eqnarray}\label{L1}\left(
          \begin{array}{cc}
            I_1 & 0 \\
           0 & 0 \\
          \end{array}
        \right)\leq \left(
          \begin{array}{cc}
            U^\ast_1U_1 & U_1^\ast I_1U_2 \\
           U_2^\ast I_1 U_1 & U_2^\ast I_1U_2
          \end{array}
        \right)\,,\end{eqnarray} which implies that $I_1\leq U_1^\ast
        U_1$. Since $U^\ast U=I$ hence \begin{eqnarray}\label{L2}\left(
          \begin{array}{cc}
            U_1^\ast U_1+U_3^\ast U_3 &U_1^\ast U_3+U_3^\ast U_4  \\
           U_2^\ast U_1+U_4^\ast U_3 & U_2^\ast U_2+U_4^\ast U_4 \\
          \end{array}
        \right)= \left(
          \begin{array}{cc}
            I_1 & 0 \\
           0 & I_2 \\
          \end{array}
        \right).\end{eqnarray}
From \eqref {L1} and \eqref {L2} we see that $I_1=U_1^\ast U_1$,
$U_3=0$ and $U_2^\ast U_1=0$. Thus $U_2^\ast=U_2^\ast U_1U_1^\ast=0$ and this ensures that $U^\ast PU=\left(\begin{array}{cc}
            U^\ast_1U_1 & 0 \\
           0 & 0 \\
          \end{array}
        \right)=P$ as desired.
\end{proof}
In the sequel we need to use the structure of the spectral family
$\{E_\lambda(A)\}$ corresponding to an operator $A \in
\mathbb{B}_h(\mathscr{H})$; cf. \cite{HEL}. Recall that
$E_\lambda(A)$ can be defined as the strong operator limit
$\varphi_\lambda(A)$ of the sequence $\{\varphi_{\lambda, n}(A)\}$,
where $\{\varphi_{\lambda, n}\}$ is a sequence of decreasing
nonnegative continuous functions on the real line  pointwise
converging to the following function defined on the spectrum ${\rm
sp}(A)$ of $A$:
$$
\varphi_\lambda(t)=
\begin{cases}
1 & \text{if } -\infty < t \leq \lambda \\
0 & \text{if } \lambda < t<\infty
\end{cases}
$$
\begin{remark}\label{U12}
It follows from Lemma \ref{U1} that if $A$ is a positive operator
and $U\in \mathcal{U}(\mathscr{H})$, then $E_\lambda(U^\ast
AU)=U^\ast E_\lambda(A)U$ for every $\lambda\in\mathbb{R}$.
\end{remark}

\begin{lemma}\label{U13}\cite[theorem 3]{OLS}
Let $A$ and $B$ be positive operators in $\mathbb{B}(\mathscr{H})$.
Then $A^n\leq B^n$ for $n\in\mathbb{N}$ if and only if
$E_\lambda(A)\leq E_\lambda(B)$ for every $\lambda\in\mathbb{R}$.
\end{lemma}
\begin{theorem}\label{PP}
Let $A$ be a positive operator and $U\in \mathcal{U}(\mathscr{H})$
such that $A^n\leq U^\ast A^nU$ for all $n\in\mathbb{N}$. Then
$A=U^\ast AU$.
\end{theorem}
\begin{proof}
From Lemma \ref{U13} and Remark \ref{U12} we see $$E_\lambda(A)\leq
E_\lambda(U^\ast AU)=U^\ast E_\lambda(A) U.$$ Thus from Lemma
\ref{U11} we have $E_\lambda(A)=U^\ast E_\lambda(A)U$, which implies
that $UE_\lambda(A)=E_\lambda(A)U$ for every $\lambda\in\mathbb{R}$.
Hence $UA=AU$, or equivalently $A=U^\ast AU$
\end{proof}
\begin{corollary}
Suppose that $A$ and $B$ be two positive operators such that
$\mathcal{K}_{n_0,A,B}$ and $\mathcal{K}_{m_0,B,A}$ are either
strongly compact or weakly closed nonempty sets for some positive
integers $n_0$ and $m_0$, respectively. Then $A$ and $B$ are
unitarily equivalent.
\end{corollary}
\begin{proof}
By Theorem \ref{SS} there exist operators $U, V \in
\mathcal{U}(\mathscr{H})$ such that $A^n\leq U^\ast B^nU$ and
$B^n\leq V^\ast A^nV$ for all $n\in\mathbb{N}$. Thus $A^n\leq U^\ast
B^nU\leq U^\ast V^\ast A^nVU$. Now the result is obtained from
Theorem \ref{PP}.
\end{proof}

\end{document}